\documentclass[a4paper, 12pt]{amsart}
\usepackage{txfonts}
\usepackage{amssymb}
\usepackage{mathrsfs}
\usepackage{latexsym}
\usepackage{cases}
\usepackage{amscd}
\usepackage{amsfonts}
\usepackage{amsmath}
\usepackage[all]{xy}
\usepackage{xy}

\newtheorem*{theorem1}{Theorem 1}
\newtheorem*{theorem2}{Theorem 2}
\newtheorem{theorem}{Theorem}[section]
\newtheorem{lemma}[theorem]{Lemma}

\theoremstyle{definition}

\theoremstyle{remark}
\newtheorem{remark}[theorem]{Remark}

\numberwithin{equation}{section}

\begin{document}

\def\temptablewidth{0.5\textwidth}

\title{Almost complex structures \\on $(n-1)$-connected $2n$-manifolds}

\author{Huijun Yang}
\address{Hua Loo-Keng Key laboratory of Mathematics, Academy of Mathematics and Systems Science,
Chinese Academy of Sciences, Beijing 100190, China}
\email{yhj@amss.ac.cn}


\begin{abstract}   Let $M$ be a closed $(n-1)$-connected $2n$-dimensional smooth
manifold with $n\geq 3$. In terms of the system of invariants for
such manifolds introduced by Wall, we obtain necessary
and sufficient conditions for $M$ to admit an almost complex
structure.
\end{abstract}

\maketitle

\section{Introduction}

First we introduce some notations. For a topological space $X$, let $Vect_{C}(X)$ (resp. $Vect_{R}(X)$)
be the set of isomorphic classes of complex (resp. real) vector
bundles on $X$, and let $r\colon Vect_{C}(X)\rightarrow Vect_{R}(X)$
be the real
reduction, which induces the real reduction homomorphism $\tilde{r}\colon \widetilde{K}%
(X)\rightarrow \widetilde{KO}(X)$ from the reduced $KU$-group to the
reduced $KO$-group of $X$.~For a map $f\colon X\rightarrow Y$
between topological spaces $X$~and~$Y$,~denote by $f_{u}^{*}\colon
\widetilde{K}(Y)\rightarrow \widetilde{K}(X)$~and~$f_{o}^{*}\colon
\widetilde{KO}(Y)\rightarrow \widetilde{KO}(X)$~the induced
homomorphisms.

Let $M$ be a $2n$-dimensional smooth manifold with tangent bundle
$TM$. We say that $M$ admits an \emph{almost complex structure}
(resp.~a \emph{stable almost complex structure}) if $TM \in
\mathrm{Im} r$ (resp.~$TM \in  \mathrm{Im} \tilde{r}$). Clearly, $M$
admits an almost complex structure implies that $M$ admits a stable
almost complex structure. It is a classical topic in geometry to
determine which $M$ admits an almost complex structure. See for
instance \cite{wu,eh,mg,su}.
In this paper we determine those closed $(n-1)$-connected $2n$-dimensional smooth manifolds $M$ with $%
n\geq 3$ that admit an almost complex structure.

Throughout this paper, $M$ will be a closed oriented
$(n-1)$-connected $2n$- dimensional smooth manifold with $ n\ge3$.
In \cite{wa}, C.T.C. Wall assigned to each $M$ a system of invariants
as follows.

\shortstack[l]{
1) $H=H^{n}(M;\mathbb{Z})\cong Hom(H_{n}(M;\mathbb{Z});\mathbb{Z}
)\cong\oplus _{j=1}^{k}\mathbb{Z}$, the cohomology group of $M$,  \\\quad with $k$ the $n$-th Betti number of $M$,
}

\shortstack[l]{
2) $I\colon H\times
H\rightarrow \mathbb{Z}$, the intersection form of $M$ which is unimodular and $n$- sym-\\\quad metric, defined by\\
\qquad \qquad\qquad\qquad\qquad $I(x,y)=<x\cup y,[M]>$,\\
\quad where
the homology class $[M]$ is the orientation class of $M$,\\
3) A map $\alpha \colon H_{n}(M;\mathbb{Z})\rightarrow \pi
_{n-1}(SO_{n})$ that assigns each element $x\in H_{n}(M;\mathbb{Z})$
to\\ \quad the
characteristic map $\alpha (x)$ for the normal bundle of the embedded $n$-sphere \\\quad$%
S_{x}^n$ representing $x$.
}

These invariants satisfy the relation (\cite[Lemma 2]{wa})
\begin{equation}\label{gx}
\alpha(x+y)=\alpha(x)+\alpha(y)+I(x,y)~\partial\iota_n,
\end{equation}
where $\partial$ is the boundary homomorphism in the exact sequence
\begin{equation}\label{tl}
\cdots\rightarrow\pi_n(S^n)\xrightarrow{\partial}\pi_{n-1}(SO_n)\xrightarrow{S}\pi_{n-1}(SO_{n+1})\rightarrow\cdots
\end{equation}
of the fiber bundle $SO_n\hookrightarrow SO_{n+1}\rightarrow S^n$,
and $\iota_n\in\pi_n(S^n)$ is the class of the identity map.

Denote by $\chi = S\circ ~\alpha\colon
H_{n}(M;\mathbb{Z})\rightarrow \pi _{n-1}(SO_{n+1})\cong
\widetilde{KO}(S^{n})$ the composition map, then from (\ref{gx}) and
(\ref{tl})
\begin{equation}\label{chi}
\chi =S\circ \alpha \in H^{n}(M;\widetilde{KO}(S^{n}))=Hom(H_{n}(M;%
\mathbb{Z});\widetilde{KO}(S^{n}))
\end{equation}
can be viewed as an $n$-dimensional cohomology class of $M$, with
coefficient in $\widetilde{KO}(S^n)$. It follows from
Kervaire~\cite[Lemma~1.1]{ke}~and~Hirzebruch index Theorem \cite{hi}
that the Pontrjagin classes $p_{j}(M)\in H^{4j}(M;\mathbb{Z})$ of
$M$ can be expressed in terms of the cohomology class $\chi$ and the
index $\tau$ of the intersection form $I$ (when $n$ is even) as
follows (cf. Wall~\cite[p.~179-180]{wa}).

\begin{lemma}{Let $M$ be a closed oriented
$(n-1)$-connected $2n$-dimensional smooth manifold with $ n\ge3$. Then
\begin{equation*} p_{j}(M)=
\begin{cases}
\pm a_{n/4}(n/2-1)!\chi, &  n\equiv
0(mod~4), j=n/4, \\
\frac{a_{n/4}^{2}}{2}((n/2-1)!)^{2}\{1-\frac{
(2^{n/2-1}-1)^{2}}{2^{n-1}-1}\binom{n}{n/2}\frac{B_{n/4}^{2}}{B_{n/2}}\}I(\chi,\chi)
\\+\frac{n!}{2^{n}(2^{n-1}-1)B_{n/2}}\tau, &  n\equiv
0(mod~4), j=n/2,\\
\frac{n!}{2^{n}(2^{n-1}-1)B_{n/2}} \tau, &  n\equiv 2(mod~4), j=n/2,
\end{cases}
\end{equation*}
where
\begin{equation*}a_{n/4}=
\begin{cases}
1, & n\equiv 0~(mod~8),\\
2, & n\equiv 4~(mod~8),
\end{cases}
\end{equation*}
$B_m$ is the $m$-th Bernoulli number.}
\end{lemma}

Now we can state the main results as follows.

\begin{theorem1} {Let $M$ be a closed oriented $(n-1)$-connected
$2n$-dimensional smooth manifold with $ n\ge3$, $\chi$ be the
cohomology class defined in (\ref{chi}), $\tau$ the index of the intersection form $I$ (when $n$ is even). Then the necessary and
sufficient conditions for $M$ to admit a stable almost complex
structure are:

\shortstack[l]{
1)~$n\equiv 2,~3,~5,~6,~7~(mod~8)$,~or\\
2)~if~$n\equiv 0~(mod~8)\colon \chi \equiv 0~(mod~2)$~and~$\frac{%
(B_{n/2}-B_{n/4})}{B_{n/2}B_{n/4}}\cdot \frac{n\tau }{2^{n}}\equiv 0~(mod%
~2)$,\\
3)~if~$n\equiv
4~(mod~8)\colon \frac{(B_{n/2}+B_{n/4})}{B_{n/2}B_{n/4}}\cdot
\frac{\tau
}{2^{n-2}}\equiv 0~(mod~2)$,\\
4)~if~$n\equiv 1~(mod~8)\colon \chi =0$.}}
\end{theorem1}

\begin{theorem2}{ Let $M$ be a closed oriented $(n-1)$-connected
$2n$-dimensional smooth manifold with $ n\ge3$, $k$ be the $n$-th
Betti number, $I$ be the intersection form, and $p_j(M)$ be the
Pontrjagin class of $M$ as in Lemma 1.1. Then $M$ admits an almost complex
structure if and only if $M$ admits a stable almost complex
structure and one of the following conditions are satisfied:

\shortstack[l]{
1)~If~$n\equiv 0~(mod~4)\colon$~$
4p_{n/2}(M)-I(p_{n/4}(M),~p_{n/4}(M))=8~(k+2)$,\\
2)~if~$n\equiv 2~(mod~8)\colon$ there exists an element~$x\in H^{n}(M;\mathbb{Z
})$ such that \\ \qquad  $x\equiv \chi ~(mod~2)$ and $I(x,
x)=(2(k+2)+p_{n/2}(M))/((n/2-1)!)^{2}$,\\
3)~if~$n\equiv 6~(mod~8)\colon$~there exists an element~$x\in H^{n}(M;\mathbb{Z
})$ such that \\ \qquad  $I(x, x)=(2(k+2)+p_{n/2}(M))/((n/2-1)!)^{2}$,\\
4)~if~$n\equiv 1~(mod~4)\colon$~$2((n-1)!)\mid(2-k)$,\\
5)~if~$n\equiv 3~(mod~4)\colon$~$(n-1)!\mid(2-k)$.}}
\end{theorem2}
\begin{remark}{ i) Since the rational numbers $\frac{(B_{n/2}-B_{n/4})}{%
B_{n/2}B_{n/4}}\cdot \frac{n\tau }{2^{n}}$ and $\frac{(B_{n/2}+B_{n/4})}{%
B_{n/2}B_{n/4}}\cdot \frac{\tau }{2^{n-2}}$ in Theorem 1 can be viewed as $2$-adic integers (see the proof of Theorem 1), it makes sense to take congruent classes modulo 2.

ii) In the cases 2) and 3) of Theorem 2, when the conditions are
satisfied, the almost complex structure on $M$ depends on the
choice of $x$.
}
\end{remark}
This paper is arranged as follows. In \S 2 we obtain presentations
for the groups $\widetilde{KO}(M)$, $\widetilde{K}(M)$ and determine
the real reduction $\tilde{r}\colon \widetilde{K}(M)\rightarrow
\widetilde{KO}(M)$ accordingly. In \S 3 we determine the expression
of $TM\in\widetilde{KO}(M)$ with respect to the presentation of
$\widetilde{KO}(M)$ obtained in \S 2. With these preliminary
results, Theorem 1 and Theorem 2 are established in \S 4.

I would like to thank my supervisor H. B. Duan, Dr. Su and Dr. Lin for their
help with the preparation of this paper.

\section{The real reduction $\tilde{r}\colon\protect\widetilde{K}(M)\rightarrow%
\protect\widetilde{KO}(M)$}

According to Wall \cite{wa}, $M$ is homotopic to
a $CW$ complex $(\vee _{\lambda=1}^{k}S_{\lambda}^{n})\cup
_{f}\mathbb{D}^{2n}$, where $k$ is the $n$-th Betti number of $M$, $\vee _{\lambda=1}^{k}S_{\lambda}^{n}$
is the wedge sum of $n$-spheres which is the $n$-skeleton of $M$ and $f \in \pi_{2n-1}(\vee
_{\lambda=1}^{k}S_{\lambda}^{n})$ is the attaching map of $\mathbb{D}^{2n}$ which is
determined by the intersection form $I$ and the map $\alpha$~(cf. Duan and Wang \cite{dw}).

Let $i\colon \vee _{\lambda=1}^{k}S_{\lambda}^{n}\rightarrow M$ be
the inclusion map of the $n$-skeleton of $M$ and $p\colon M\rightarrow S^{2n}$ be the
map collapsing the $n$-skeleton $\vee _{\lambda=1}^{k}S_{\lambda}^{n}$ to the base point. Then by the
naturality of the Puppe sequence, we have the following exact ladder:

\begin{table}[!htbp]
\begin{tabular}[b]{c@{\hspace{4pt}}c@{\hspace{4pt}}c@{\hspace{4pt}}c@{\hspace{4pt}}c@{\hspace{4pt}}c@{\hspace{4pt}}c@{\hspace{4pt}}c@{\hspace{4pt}}c@{\hspace{4pt}}c}
&$\widetilde{K}(\vee_{\lambda=1}^k S_{\lambda}^{n+1})$ & $\smallskip
\smallskip \overset{\Sigma f_{u}^{\ast }}{\rightarrow } $&
$\widetilde{K}(S^{2n})$ & $\overset{p_u^{\ast }}{ \rightarrow }$ &
$\widetilde{K}(M)$ & $\overset{i_u^{\ast }}{\rightarrow }$ &
$\widetilde{K}(\vee_{\lambda=1}^k S_{\lambda}^{n})$ &
$\overset{f_u^{\ast }}{\rightarrow }$ &
$\widetilde{K}(S^{2n-1})$ \\
(2.1)\quad\quad&$\tilde{r}\downarrow$ &  & $\tilde{r}\downarrow$ &  & $\tilde{r}%
\downarrow$ &  & $\tilde{r}\downarrow$ &  & $\tilde{r}\downarrow $\\
&$\widetilde{KO}(\vee_{\lambda=1}^k S_{\lambda}^{n+1})$ & $\overset{\Sigma f_{o}^{\ast }}{%
\rightarrow }$ & $\widetilde{KO}(S^{2n})$ & $\overset{p_o^{\ast
}}{\rightarrow }$
& $\widetilde{KO}(M)$ & $\overset{i_o^{\ast }}{\rightarrow }$ & $\widetilde{KO}%
(\vee_{\lambda=1}^k S_{\lambda}^{n})$ & $\overset{f_o^{\ast }}{\rightarrow }$ & $\widetilde{KO}%
(S^{2n-1})$
\end{tabular}
\end{table}
where the horizontal homomorphisms $\Sigma f_{u}^{\ast}$, $\Sigma
f_{o}^{\ast }$, $p_u^{\ast}$, $p_o^{\ast }$, $i_u^{\ast}$,
$i_o^{\ast }$ and $f_{u}^{\ast }$, $f_{o}^{\ast }$ are induced by
$\Sigma f$, $p$, $i$ and $f$ respectively, and where $\Sigma$
denotes the suspension.

Let $\mathbb{Z}\beta$ (resp. $\mathbb{Z}_2\beta$) be the infinite
cyclic group (resp. finite cyclic group of order 2) generated by
$\beta$. Then the generators $\omega_{C}^m$ (resp. $\omega_{R}^m$) of the cyclic group $\widetilde{K}(S^m)$ (resp. $\widetilde{KO}(S^m)$) with $m > 0$ can be so chosen such that the
real reduction $\tilde{r}\colon \widetilde{K} (S^m)\rightarrow
\widetilde{KO}(S^m)$ can be summarized as in Table 1 (cf. Mimura~and~Toda~\cite[Theorem~6.1,~p.~211]{mt}).

\begin{table}[!htbp]
\centering
\begin{tabular}[b]{llll}
\multicolumn{4}{c}{Table~1.~ Real reduction~$\tilde{r}\colon \widetilde{K}
(S^m)\rightarrow
\widetilde{KO}(S^m)$}\\[0pt] \hline
\multicolumn{1}{|c}{\rule{0pt}{13pt} $m~(mod~8)$} &
\multicolumn{1}{|c}{$\widetilde{K}(S^{m})$} &
\multicolumn{1}{|c}{$\widetilde{KO}(S^{m})$} &
\multicolumn{1}{|c|}{$\tilde{r}\colon \widetilde{K} (S^m)\rightarrow
\widetilde{KO}(S^m)$}
\\[0pt]
\hline

\multicolumn{1}{|c}{$0$} &
\multicolumn{1}{|l}{$\mathbb{Z}\omega_{C}^m$} &
\multicolumn{1}{|l}{$\mathbb{Z}\omega_{R}^m$} &
\multicolumn{1}{|l|}{$\tilde{r}(\omega_{C}^m)=2\omega_{R}^m$}
\\ \hline

\multicolumn{1}{|c}{$1$} & \multicolumn{1}{|l}{$0$} &
\multicolumn{1}{|l}{$\mathbb{Z}_{2}\omega_{R}^m$} & \multicolumn{1}{|l|}{$\tilde{r}=0$} \\
\hline

\multicolumn{1}{|c}{$2$} &
\multicolumn{1}{|l}{$\mathbb{Z}\omega_{C}^m$} &
\multicolumn{1}{|l}{$\mathbb{Z}_{2}\omega_{R}^m$} &
\multicolumn{1}{|l|}{$\tilde{r}(\omega_{C}^m)=\omega_{R}^m$}
\\ \hline

\multicolumn{1}{|c}{$4$} &
\multicolumn{1}{|l}{$\mathbb{Z}\omega_{C}^m$} &
\multicolumn{1}{|l}{$\mathbb{Z}\omega_{R}^m$} & \multicolumn{1}{|l|}{$\tilde{r}(\omega_{C}^m)=\omega_{R}^m$} \\
\hline

\multicolumn{1}{|c}{$6$} &
\multicolumn{1}{|l}{$\mathbb{Z}\omega_{C}^m$} &
\multicolumn{1}{|l}{$0$} & \multicolumn{1}{|l|}{$\tilde{r}=0$} \\
\hline

\multicolumn{1}{|c}{$3$, $5$, $7$} & \multicolumn{1}{|l}{$0$} &
\multicolumn{1}{|l}{$0$} & \multicolumn{1}{|l|}{$\tilde{r}=0$} \\
\hline
\end{tabular}
\end{table}

Denoted by $t_{ju}^{\ast}\colon
\widetilde{K}(S_j^n)\rightarrow\widetilde{K}(\vee_{\lambda=1}^kS_{\lambda}^n)$~and~$t_{jo}^{\ast}\colon
\widetilde{KO}(S_j^n)\rightarrow\widetilde{KO}(\vee_{\lambda=1}^kS_{\lambda}^n)$
the homomorphisms induced by $t_{j}\colon
\vee_{\lambda=1}^kS_{\lambda}^n\rightarrow S_{j}^n$ which collapses
$\vee_{\lambda\neq j}S_{\lambda}^{n}$ to the base point. Then we
have:

\begin{lemma}{ Let $M$ be a closed oriented
$(n-1)$-connected $2n$-dimensional smooth manifold with $ n\ge3$. Then the presentations of the groups
$\widetilde{K}(M)$ and $
\widetilde{KO}(M)$ as well as the real reduction $\tilde{r}\colon \widetilde{K}%
(M)\rightarrow \widetilde{KO}(M)$ can be given as in Table 2.
\begin{table}[!htbp]
\centering
\begin{tabular}{llll} \multicolumn{4}{c}{Table~2.~Real reduction~$\tilde{r}\colon \widetilde{K}
(M)\rightarrow\widetilde{KO}(M)$}\\[0pt] \hline

\multicolumn{1}{|c}{\rule{0pt}{14pt} $n~(mod\text{ }8)$} &
\multicolumn{1}{|c}{$\widetilde{K}(M)$} &
\multicolumn{1}{|c}{$\widetilde{KO}(M)$} &
\multicolumn{1}{|c|}{$\tilde{r}\colon \widetilde{K} (M)\rightarrow
\widetilde{KO}(M)$} \\
\hline

\multicolumn{1}{|c}{$0$} &
\multicolumn{1}{|l}{$\mathbb{Z}\xi\oplus \bigoplus
_{j=1}^{k}\mathbb{Z}\eta _{j}$} & \multicolumn{1}{|l}{$\mathbb{Z}%
\gamma \oplus \bigoplus _{j=1}^{k}\mathbb{Z}\zeta _{j}$} &
\multicolumn{1}{|l|}{$\tilde{r}(\xi )=2\gamma $, $\tilde{r}(\eta
_{j})=2\zeta _{j}$}
\\ \hline

\multicolumn{1}{|c}{ $1$} &
\multicolumn{1}{|l}{$\mathbb{Z}\xi$} &
\multicolumn{1}{|l}{$\mathbb{Z}_2
\gamma \oplus \bigoplus _{j=1}^{k}\mathbb{Z}_2\zeta _{j}$} & \multicolumn{1}{|l|}{$\tilde{r}(\xi )=\gamma$} \\
\hline

\multicolumn{1}{|c}{ $2$} &
\multicolumn{1}{|l}{$\mathbb{Z}\xi\oplus \bigoplus
_{j=1}^{k}\mathbb{Z}\eta _{j}$} & \multicolumn{1}{|l}{$\mathbb{Z}%
\gamma \oplus \bigoplus _{j=1}^{k}\mathbb{Z}_2\zeta _{j}$} &
\multicolumn{1}{|l|}{$\tilde{r}(\xi )=\gamma $, $\tilde{r}(\eta
_{j})=\zeta _{j}$}
\\ \hline

\multicolumn{1}{|c}{ $4$} &
\multicolumn{1}{|l}{$\mathbb{Z}\xi\oplus \bigoplus
_{j=1}^{k}\mathbb{Z}\eta _{j}$} &
\multicolumn{1}{|l}{$\mathbb{Z}%
\gamma \oplus \bigoplus _{j=1}^{k}\mathbb{Z}\zeta _{j}$} &
\multicolumn{1}{|l|}{$\tilde{r}(\xi)=2\gamma $, $\tilde{r}(\eta _{j})=\zeta _{j}$} \\
\hline

\multicolumn{1}{|c}{ $5$} &
\multicolumn{1}{|l}{$\mathbb{Z}\xi$} &
\multicolumn{1}{|l}{$\mathbb{Z}_{2}\gamma$} &
\multicolumn{1}{|l|}{$\tilde{r}(\xi )=\gamma$} \\ \hline

\multicolumn{1}{|c}{$6$} &
\multicolumn{1}{|l}{$\mathbb{Z}\xi\oplus \bigoplus
_{j=1}^{k}\mathbb{Z}\eta _{j}$} & \multicolumn{1}{|l}{$\mathbb{Z}%
\gamma$} & \multicolumn{1}{|l|}{$\tilde{r}(\xi )=\gamma$,
$\tilde{r}(\eta _{j})=0$} \\ \hline

\multicolumn{1}{|c}{ $3$, $7$} &
\multicolumn{1}{|l}{$\mathbb{Z}\xi$} &
\multicolumn{1}{|l}{$0$} & \multicolumn{1}{|l|}{$\tilde{r}=0$} \\\hline
\end{tabular}
\end{table}\\
where the generators $\xi$,~$\eta_j$,~$\gamma$,~$\zeta_j$,~$1\le j \le
k$, satisfy:
\begin{equation*}
\left\{\begin{aligned}\xi=p_{u}^{\ast}(\omega_{C}^{2n}),~
i_{u}^{\ast}(\eta_j)=t_{ju}^{\ast}(\omega_{C}^n);\\
\gamma=p_{o}^{\ast}(\omega_{R}^{2n}),~
i_{o}^{\ast}(\zeta_j)=t_{jo}^{\ast}(\omega_{R}^n).
\end{aligned}\right.
\end{equation*}
}
\end{lemma}

\begin{proof}{ We assert that

a) the induced homomorphisms $f^{\ast }_u$, $f^{\ast }_o$, $\Sigma
f^{\ast }_u$ and $\Sigma f^{\ast }_o$ in (2.1) are trivial, moreover,

b) the short exact sequences
\begin{gather}
0\rightarrow\widetilde{K}(S^{2n})\xrightarrow{p_u^{\ast
}}\widetilde{K}(M)\xrightarrow{i_u^{\ast
}}\widetilde{K}(\vee_{\lambda=1}^k S_{\lambda}^{n})\rightarrow0\tag{i} \notag \\
0\rightarrow\widetilde{KO}(S^{2n})\xrightarrow{p_o^{\ast
}}\widetilde{KO}(M)\xrightarrow{i_o^{\ast
}}\widetilde{KO}(\vee_{\lambda=1}^k S_{\lambda}^{n})\rightarrow0\tag{ii} \notag
\end{gather}
split.


Denote by $c\colon \widetilde{KO}(X)\rightarrow\widetilde{K}(X)$ the
complexification. Then by (2.1), combining these assertions with
the fact that $\tilde{r}\circ c=2$, all the results in Table
2 are easily verified.

Now we prove assertions a) and b).

Firstly, by the Bott periodicity Theorem \cite{bo}, we may assume
that the horizontal homomorphisms $\Sigma f_{u}^{\ast}$, $\Sigma
f_{o}^{\ast }$, $p_u^{\ast}$, $p_o^{\ast }$, $i_u^{\ast}$,
$i_o^{\ast }$ and $f_{u}^{\ast }$, $f_{o}^{\ast }$ in (2.1) are
induced by $\Sigma^9f$, $\Sigma^8p$, $\Sigma^8i$ and $\Sigma^8f$
respectively, where $\Sigma^j$ denotes the $j$-th iterated
suspension. Note that $\Sigma^9f\in \pi_{2n+8}(\vee
_{\lambda=1}^{k}S_{\lambda}^{n+9})$ and $\Sigma^8f\in
\pi_{2n+7}(\vee _{\lambda=1}^{k}S_{\lambda}^{n+8})$, and the groups
$\pi_{2n+8}(\vee _{\lambda=1}^{k}S_{\lambda}^{n+9})$ and
$\pi_{2n+7}(\vee _{\lambda=1}^{k}S_{\lambda}^{n+8})$ are all in
their stable range, that is $\pi_{2n+8}(\vee
_{\lambda=1}^{k}S_{\lambda}^{n+9})\cong \pi_{2n+7}(\vee
_{\lambda=1}^{k}S_{\lambda}^{n+8})\cong
\oplus_{\lambda=1}^{k}\pi_{n-1}^s$, where $\pi_{n-1}^S$ is the
$(n-1)$-th stable homotopy group of spheres. Thus the fact that
$\Sigma f_{u}^{\ast}$ and $f^{\ast }_u$ are trivial can be deduced
easily from Table 1 and Adams \cite[proposition 7.1]{ad}; the fact
that $\Sigma f^{\ast }_o$ and $f^{\ast }_o$ are trivial when
$n\nequiv{1}~(mod~8)$ follows from Table 1 while the fact that
$\Sigma f^{\ast }_o$ and $f^{\ast }_o$ are trivial when
$n\equiv~1~(mod~8)$ follows from Adams \cite[proposition 7.1]{ad}.
This proves assertion a).

Secondly, (i) of assertion b) is true since
the abelian group $\widetilde{K}(\vee_{\lambda=1}^k
S_{\lambda}^{n})$ is free.

Finally we prove (ii) of assertion b). For the cases
$n\nequiv{1,2}~(mod~8)$ the proof is similar to (i).

Case $n\equiv 1~(mod~8)$. From (2.1), Table 1 and (i) we get that $\widetilde{K}(M)\cong\mathbb{Z}$ and
$\widetilde{KO}(M)$ is a finite group. Therefore, for each $x\in
\widetilde{KO}(M)$, we have $2x=\tilde{r}\circ c(x)=0$, which
implies (ii) of assertion b) in this case.

Case $n\equiv 2~(mod~8)$. By (i), we may write $\widetilde{K}(M)$ as
\begin{equation*}
\widetilde{K}(M)=\mathbb{Z}\xi\oplus \bigoplus
_{j=1}^{k}\mathbb{Z}\eta _{j},
\end{equation*}
where the generators $\xi$,
$\eta_j$, $1\le j\le k$, satisfy
$\xi=p_{u}^{\ast}(\omega_{C}^{2n})$,
$i_{u}^{\ast}(\eta_j)=t_{ju}^{\ast}(\omega_{C}^n)$. By Hilton-Milnor theorem \cite[p. 511]{wh} we know that the group $\pi_{2n-1}(\vee_{j=1}^kS_j^{n})$ can be decomposed as:
\begin{equation*}
\pi_{2n-1}(\vee_{j=1}^kS_j^{n})\cong\oplus_{j=1}^k \pi_{2n-1}(S_{j}^n)\oplus_{1\le i < j\le k}\pi_{2n-1}(S_{ij}^{2n-1}),
\end{equation*}
where $S_{ij}^{2n-1}=S^{2n-1}$, the group $\pi_{2n-1}(S_{j}^n)$ is embedded in $\pi_{2n-1}(\vee_{j=1}^kS_j^{n})$ by the natural inclusion, and the group $\pi_{2n-1}(S_{ij}^{2n-1})$ is embedded by composition with the Whitehead product of certain elements in $\pi_{n}(\vee_{j=1}^kS_j^{n})$. Hence by Duan and Wang \cite[Lemma 3]{dw}, the attaching map $f$ can be decomposed accordingly as:
\begin{equation*}
f=\Sigma_{j=1}^k f_{j} + g,
\end{equation*}
where
\begin{equation*}
f_j\in Im J \subset \pi_{2n-1}(S^n)
\end{equation*}
$J$ being the $J$-homomorphism and $g\in\oplus_{1\le i< j\le
k}\pi_{2n-1}(S_{ij}^{2n-1})$. Moreover, since the suspension of the Whitehead
product is trivial, it follows that the homotopy group
$\pi_{2n+7}(\vee_{j=1}^kS_j^{n+8})$ can be decomposed as:
\begin{equation*}
\pi_{2n+7}(\vee_{j=1}^kS_j^{n+8})\cong \oplus_{j=1}^k
\pi_{2n+7}(S_j^{n+8}),
\end{equation*}
and accordingly $\Sigma^8f$ can be decomposed as:
\begin{equation*}
\Sigma^8f=\oplus_{j=1}^k \Sigma^8 f_j \in\oplus_{j=1}^k
\pi_{2n+7}(S_j^{n+8})
\end{equation*}
with
\begin{equation*}
\Sigma^{8} f_j\in Im J \subset \pi_{2n+7}(S_j^{n+8})\cong \pi_{n-1}^s.
\end{equation*}
Denote by
$e_{C}(\Sigma^8 f_j)$ the $e_{C}$ invariant of $\Sigma^8 f_j$ defined in Adams \cite{ad}, $\Psi_{C}^{-1}\colon\widetilde{K}(M)\rightarrow\widetilde{K}(M)$~and
~$\Psi_{R}^{-1}=id\colon\widetilde{KO}(M)\rightarrow\widetilde{KO}(M)$~ the Adams operations,
~where~$id$~is the identity map.
Then it follows from Adams \cite[Proposition 7.19]{ad} that
\begin{equation*}
e_{C}(\Sigma^8 f_j)=0,
\end{equation*}
for each $1\le j\le k$.
Hence, by considering the map
\begin{equation*}
\tilde{t}_j\colon(\vee _{\lambda=1}^{k}S_{\lambda}^{n+8})\cup
_{\Sigma^8f}\mathbb{D}^{2n+8}\rightarrow S_{j}^{n+8}\cup
_{\Sigma^8f_j}\mathbb{D}^{2n+8}
\end{equation*}
which collapses $\vee_{\lambda\neq j}S_{\lambda}^{n+8}$ to a point, it's easy to see from \cite[proposition 7.5, Proposition 7.8]{ad} and the naturality of Adams operation that
\begin{equation*}
\Psi_{C}^{-1}(\eta_j)=(-1)^{n/2}\eta_j+l\cdot((-1)^n-(-1)^{n/2})\xi\in\widetilde{K}(M)
\end{equation*}
for each $\eta_j$, and for some $l\in\mathbb{Z}$.
Therefore from
\begin{equation*}
\tilde{r}\circ\Psi_{C}^{-1}=\Psi_{R}^{-1}\circ\tilde{r},
\end{equation*}
we have
\begin{equation*}
\Psi_{R}^{-1}\tilde{r}(\eta_j)=-\tilde{r}(\eta_j)+2l\tilde{r}(\xi).
\end{equation*}
That is
\begin{equation*}
2\tilde{r}(\eta_j-l\xi)=0.
\end{equation*}
But from (2.1) and Table 1, we get
\begin{equation*}
i_{o}^{\ast}\tilde{r}(\eta_j-l\xi)=t_{oj}^{\ast}(\omega_{R}^n).
\end{equation*}
That is
\begin{equation*}
\tilde{r}(\eta_j-l\xi)\neq0\in\widetilde{KO}(M).
\end{equation*}
Thus (ii) of assertion b) in this case is established and the proof
is finished. }
\end{proof}
\begin{remark}\label{re}{ Since the induced homomorphisms $i^{\ast}\colon H^{n}(M;\mathbb{Z})\rightarrow%
H^{n}(\vee _{\lambda=1}^{k}S_{\lambda}^{n};\mathbb{Z})$ and $p^{\ast}\colon H^{2n}(S^{2n};\mathbb{Z})%
\rightarrow H^{2n}(M;\mathbb{Z})$ are both isomorphisms, and the
generator $\omega_{C}^{2n}\in\widetilde{K}(S^{2n})$ can be chosen
such that its $n$-th chern class $c_n(\omega_{C}^{2n})=(n-1)!$(cf. Hatcher \cite[p. 101]{ha}), from the naturality of the chern class, we get
\begin{equation*}c_{i}(\xi)=
\begin{cases}
(n-1)!, & i=n,\\
0, &\text{others}.
\end{cases}
\end{equation*}
Similarly,
when $n$ is even, $\eta_j$, $1\le j\le k$, can be chosen such that
\begin{align*}
c_{n/2}(\Sigma_{j=1}^kx_j\eta_j)=(n/2-1)!(x_1,x_2,...,x_k)\in
H^n(M;\mathbb{Z}),
\end{align*}
where $x_j\in\mathbb{Z}$ for all $1\le j\le k$ (since $H^n(M;\mathbb{Z})\cong\oplus_{j=1}^k\mathbb{Z}$, we can write an element $x\in H^n(M;\mathbb{Z})$, under the isomorphism $i^*$, as the form $(x_1,x_2,...,x_k)$ ).}
\end{remark}

\begin{remark}{ As in Remark \ref{re}, if we write $\chi$ as $(\chi_1,...,\chi_k)\in
H^{n}(M;\widetilde{KO}(S^{n}))$, where
\begin{equation*}\chi_j\in\widetilde{KO}(S^{n})\cong
\begin{cases}
\mathbb{Z}, & n\equiv 0~(mod~4),\\
\mathbb{Z}_2, & n\equiv 1,2~(mod~8),\\
0, &\text{others},
\end{cases}
\end{equation*}
then since the tangent bundle of
sphere is stably trivial, it follows that
\begin{align*}
i_o^{\ast }(TM)=\Sigma_{j=1}^k\chi_jt_{jo}^{\ast}(\omega_{R}^n).
\end{align*}}
\end{remark}

\section{The tangent bundle of $M$}
Denote by $\dim_{c}\alpha$ the dimension of $\alpha\in Vect_{C}(M)$.
When $n\equiv0~(mod~4)$, we set
\begin{align*}
\hat{A}(M)&=~ <\hat{\mathfrak{A}}(M),[M]>, \\
\hat{A}_{C}(M)&= ~<ch(TM\otimes \mathbb{C})\cdot
\hat{\mathfrak{A}}(M),[M]>,\\
\hat{A}_{\chi}(M)&=~ <ch(\Sigma_{j=1}^{k}\chi_j\eta_j)\cdot
\hat{\mathfrak{A}}(M),[M]>,
\end{align*} where $ch$ denotes the
chern character, and $\hat{\mathfrak{A}}(M)$ is the
$\mathfrak{A}$-class of $M$ (cf. Atiyah and Hirzebruch \cite{ah}).
It follows from the differentiable Riemann-Roch theorem (cf. Atiyah
and Hirzebruch \cite{ah}) that $\hat{A}(M)$, $\hat{A}_{C}(M)$ and
$\hat{A}_{\chi}(M)$ are all integers. In particular,
$\hat{A}_{\chi}(M)$ is even when $\chi\equiv0~(mod~2)$.

Using the notation above, we get
\begin{lemma}{ Let $M$ be a closed oriented
$(n-1)$-connected $2n$-dimensional smooth manifold with $ n\ge3$. Then $TM$ can be expressed by the generators $\gamma$,~$\zeta_j$,~$1\le j\le
k$ of $\widetilde{KO}(M)$ as
follows:
\begin{equation*}TM=
\begin{cases}
l\gamma+\Sigma_{j=1}^k \chi_j\zeta_j, & n\equiv0, 2, 4~(mod~8),\\
l\gamma, & n\equiv6~(mod~8),\\
\Sigma_{j=1}^k \chi_j\zeta_j, & n\equiv1~(mod~8),\\
0, & n\equiv3,5,7~(mod~8),
\end{cases}
\end{equation*}
where
\begin{equation*}l=
\begin{cases}
\hat{A}_{C}(M)+(\Sigma_{j=1}^ka_{n/4}\chi_j\dim_{c}\eta_j-2n)\hat{A}(M)-a_{n/4}\hat{A}_{\chi}(M),
& n\equiv0~(mod~4),\\
-\frac{p_{n/2}(M)}{2((n-1)!)}, & n\equiv2~(mod~4).
\end{cases}
\end{equation*}}
\end{lemma}
\begin{proof}{ Case $n\equiv0~(mod~8)$. By Remark 2.3, we may
suppose that
\begin{equation*}
TM=l\gamma+\Sigma_{j=1}^k \chi_j\zeta_j\in\widetilde{KO}(M),
\end{equation*}
where $l\in\mathbb{Z}$. Hence from $\tilde{r}\circ c=2$ and Table 2, we have
\begin{align}\label{c}
c(TM)&=TM\otimes\mathbb{C}\\
&=l\xi+\Sigma_{j=1}^k
\chi_j\eta_j\in\widetilde{K}(M).\notag
\end{align}
Now if we regard $\xi$ and $\chi_j$ as complex vector bundles, then
from (\ref{c}) we have
\begin{equation*}
TM\otimes\mathbb{C}\oplus\varepsilon^s\cong l\xi\oplus\bigoplus_{j=1}^k\chi_j\eta_j\oplus\varepsilon^t,
\end{equation*}
for some $s,~t\in\mathbb{Z}$ satisfying
\begin{equation*}
s-t~=~l\cdot
\dim_{c}\xi+\Sigma_{j=1}^k\chi_j\dim_{c}\eta_j-2n,
\end{equation*}
where $\varepsilon^j$ is the trivial complex vector bundle of dimension $j$. Thus we have
\begin{align*}
\hat{A}_{C}(M)=&-(l\cdot
\dim_{c}\xi+\Sigma_{j=1}^k\chi_j\dim_{c}\eta_j-2n)\hat{A}(M)\\
&\text{ } +<ch(l\xi+\Sigma_{j=1}^k \chi_j\eta_j)\cdot
\hat{\mathfrak{A}}(M),[M]>,
\end{align*}
that is
\begin{equation*}
l=\hat{A}_{C}(M)+(\Sigma_{j=1}^k\chi_j\dim_{c}\eta_j-2n)\hat{A}(M)-\hat{A}_{\chi}(M).
\end{equation*}
Cases $n\equiv2, 4, 6~(mod~8)$ can be proved by the same way
as above. Note that in the case
$n\equiv2~(mod~4)$ the calculation of $\hat{A}_{C}(M)$ is
replaced by the calculation of the $n$-th chern class of $TM\otimes\mathbb{C}$.

Case $n\equiv1~(mod~4)$. From~Milnor~and~Kervaire~\cite[Lemma~1]{mk}~and~Adams~\cite[Theorem~1.3]{ad}, we get~that~$\chi=0$~implies ~$TM=0\in\widetilde{KO}(M)$. Then

i) case $n\equiv5~(mod~8)$. $TM=0\in\widetilde{KO}(M)$ because
$\chi=0$ in this case.

ii) case $n\equiv1~(mod~8)$. By Remark 2.3, we may suppose that
\begin{align*}
TM=l\gamma+\Sigma_{j=1}^k \chi_j\zeta_j,
\end{align*}where
$l\in\mathbb{Z}_2$. Then if $\chi=0$, we have $l=0$ because $TM=0$.
If $\chi\neq0$ and $l\neq0$, suppose that $\chi_{\lambda}\neq0$ for
some $1\le \lambda \le k$, set
\begin{equation*}\zeta_{j}^{\prime}=
\begin{cases}
\zeta_{j} & \text{if } j\neq\lambda, \\
\zeta_{j}+\gamma & \text{if } j=\lambda. \\
\end{cases}
\end{equation*}
Hence $\gamma$, $\zeta_{j}^{\prime}$, $1\le j\le k$, which satisfy
the conditions in Lemma 2.1, are also the generators of
$\widetilde{KO}(M)$, and we have $TM=\Sigma_{j=1}^k
\chi_j\zeta_j^{\prime}$. This implies that the generators $\gamma$,
$\zeta_{j}$, $1\le j\le k$, of $\widetilde{KO}(M)$ in Lemma 2.1 can
always be chosen such that $TM=\Sigma_{j=1}^k \chi_j\zeta_j$.

Case $n\equiv3~(mod~4)$. $TM=0$ because $\widetilde{KO}(M)=0$
in this case.
}
\end{proof}

\section{Almost complex structure on $M$}
We are now ready to prove Theorem 1 and Theorem 2.

\begin{proof}[Proof of Theorem 1.]{ Cases 1) and 2) $n\equiv0~(mod~4)$. In these cases, we get that (cf. Wall~\cite[p.~179-180]{wa})
\begin{align*}
\hat{A}(M)& =-\frac{B_{n/2}}{2(n!)}p_{n/2}(M)+\frac{1}{2}\{\frac{B_{n/4}^{2}%
}{4((n/2)!)^{2}}+\frac{B_{n/2}}{2(n!)}\}I(p_{n/4}(M),p_{n/4}(M)), \\
\hat{\mathfrak{A}}(M)&
=1-\frac{B_{n/4}}{2((n/2)!)}p_{n/4}(M)+\hat{A}(M),\\
ch(TM\otimes \mathbb{C})& =2n+(-1)^{n/4+1}\frac{p_{n/4}(M)}{(n/2-1)!}+\frac{%
I(p_{n/4}(M), p_{n/4}(M))-2p_{n/2}(M)}{2( (n-1)!)}.
\end{align*}
Hence by Lemma 1.1 we have
\begin{align}\label{a}
\hat{A}_{C}(M)= \text{ }& 2n\{1+\frac{1}{B_{n/2}}+\frac{(2^{n-1}-1)}{(2^{n/2}-1)^{2}}%
\cdot \frac{(-1)^{n/4}B_{n/2}-B_{n/4}}{B_{n/2}B_{n/4}}\}\hat{A}(M)   \\
& +\frac{1}{(2^{n/2}-1)^{2}}\cdot \frac{(-1)^{n/4}B_{n/2}-B_{n/4}}{%
B_{n/2}B_{n/4}}\cdot \frac{n\tau }{2^{n}}.\notag
\end{align}%
Moreover since the denominator of $B_{m}$, when written as the most simple fraction, is always square free and divisible
by 2 (cf. Milnor~\cite[p.
284]{ms}), we may set $B_m~=~b_m/(2c_m)$, where $c_{m}$ and $b_{m}$ are odd integers. Then multiply each side of (\ref{a}) by $(2^{n/2}-1)^2\cdot b_{n/2} \cdot b_{n/4}$, we get that
\begin{align*}
(2^{n/2}-1)^2 b_{n/2}  b_{n/4} \hat{A}_{C}(M)=&~2n\{(2^{n/2}-1)^2\cdot b_{n/2} \cdot b_{n/4}
+ 2 (2^{n/2}-1)^2  b_{n/4} c_{n/2}\\
& + 2(2^{n-1}-1)((-1)^{n/4}b_{n/2}c_{n/4}-b_{n/4}c_{n/2})\}\hat{A}(M)\\
& +2((-1)^{n/4}b_{n/2}c_{n/4}-b_{n/4}c_{n/2})\frac{n\tau}{2^n}.
\end{align*}
Since $\hat{A}_{C}(M)$ and $\hat{A}(M)$ are integers and $(2^{n/2}-1)^2\cdot b_{n/2} \cdot b_{n/4}$ is an odd integer, it follows that 
$$\frac{(-1)^{n/4}B_{n/2}-B_{n/4}}{%
B_{n/2}B_{n/4}}\cdot \frac{n\tau }{2^{n}}$$ 
is a 2-adic integer, and hence
\begin{align*}
\hat{A}_{C}(M)\equiv 0~(mod~ 2)\iff \frac{(-1)^{n/4}B_{n/2}-B_{n/4}}{%
B_{n/2}B_{n/4}}\cdot \frac{n\tau }{2^{n}}\equiv 0~(mod~2) .
\end{align*}
Then by combining these facts with Lemma 2.1 and Lemma 3.1, one verifies the
results in these cases.

Cases 3) and 4) $n\nequiv{0}~(mod~4)$ can be deduced easily
from Lemma 2.1 and Lemma 3.1. }
\end{proof}

To prove Theorem 2, we need the following lemma (see Sutherland~\cite{su}
for the proof).

\begin{lemma}{Let $N$ be a closed smooth $2n$%
-manifold. Then $N$ admits an almost complex structure if and only if it admits a stable almost complex
structure $\alpha$  satisfying
$c_{n}(\alpha )=e(N)$, where $e(N)$ is the Euler class of $N$.}
\end{lemma}

\begin{proof}[Proof of Theorem 2.]{Firstly,~it follows from Lemma~4.1 that~$M$~admits an almost complex structure if and only if there exists an element~$\alpha\in\widetilde{K}(M)$~such that
\begin{equation}
\label{jf} \left\{ \begin{aligned} \tilde{r}(\alpha)&=TM
\in\widetilde{KO}(M),\\
c_n(\alpha)&=e(M).
\end{aligned} \right.
\end{equation}
Secondly, if there exists an element~$\alpha \in \widetilde{K}(M)$~such that~$\tilde{r}(\alpha
)=TM\in
\widetilde{KO}(M)$,~then we have the following identity~(cf.~Milnor~\cite[p.~177]{ms}):
\begin{equation}\label{ch}
(\sum_j(-1)^jc_j(\alpha))\cdot(\sum_jc_j(\alpha))=\sum_j(-1)^jp_j(M).
\end{equation}
Now we prove Theorem 2 case by case.

Case 1) $n\equiv0~(mod~4)$. In this case $e(M)=k+2$.~From Lemma 4.1
we know that $M$ admits an almost complex structure if and only if
there exists an element~$\alpha\in\widetilde{K}(M)$~such that
(\ref{jf}) is satisfied. Now (\ref{ch}) becomes
\begin{equation*}
(1+c_{n/2}(\alpha)+c_{n}(\alpha))\cdot(1+c_{n/2}(\alpha)+c_{n}(\alpha))=1+(-1)^{n/4}p_{n/4}(M)+p_{n/2}(M),
\end{equation*}
it follows that
\begin{equation*}
c_{n/2}(\alpha)=(-1)^{n/4}\frac{1}{2}p_{n/4}(M),
\end{equation*}
hence
\begin{equation*}
c_{n}(\alpha)=\frac{1}{2}p_{n/2}(M)-\frac{1}{8}I(p_{n/4}(M),~p_{n/4}(M)).
\end{equation*}
Therefore from~(\ref{jf})~we get that,~$M$
admits an almost complex structure if and only if ~$M$~admits a stable almost complex structure and satisfies
\begin{equation*}
4p_{n/2}(M)-I(p_{n/4}(M), p_{n/4}(M))=8(k+2).
\end{equation*}

Case 2) $n\equiv2~(mod~8)$. In this case $e(M)=k+2$.~Set $\alpha
=l\xi+\Sigma_{j=1}^kx_j\eta_j\in\widetilde{K}(M)$ where
$l\in\mathbb{Z}$ is the integer as in Lemma 3.1 and
$x_j\in\mathbb{Z}$,~such that~$x_j\equiv\chi_j~(mod~2)$. Then from
Lemma 2.1 and Lemma 3.1, we know that
$\tilde{r}(\alpha)=TM\in\widetilde{KO}(M)$. Hence by~(\ref{jf}), we
see that $M$~admits an almost complex structure if and only if
\begin{equation*}
\left\{ \begin{aligned} &\alpha =l\xi+\Sigma_{j=1}^kx_j\eta_j\in\widetilde{K}(M),\\
&c_n(\alpha)=e(M).
\end{aligned} \right.
\end{equation*}
Let~$x=(x_1,x_2,...,x_k)\in H^n(M;\mathbb{Z})$.~Then by Remark~2.2
\begin{equation*}
c_{n/2}(\alpha)=(n/2-1)!x.
\end{equation*}
Now (\ref{ch}) is
\begin{equation*}
(1-c_{n/2}(\alpha)+c_{n}(\alpha))\cdot(1+c_{n/2}(\alpha)+c_{n}(\alpha))=1-p_{n/2}(M),
\end{equation*}
therefore
\begin{align*}
c_{n}(\alpha)&=\frac{1}{2}(I(c_{n/2}(\alpha),c_{n/2}(\alpha))-p_{n/2}(M))\\
&=\frac{1}{2}\{((n/2-1)!)^{2}I(x,~x)-p_{n/2}(M)\}.
\end{align*}
Thus it follows from ~(\ref{jf}) that~$M$~admits an almost complex structure if and only if there exists an element~$x\in
H^n(M;\mathbb{Z})$~such that
\begin{equation*}
\left\{ \begin{aligned}&x\equiv\chi~(mod~2),\\
&I(x, x)=(2(k+2)+p_{n/2}(M))/((n/2-1)!)^{2}.
\end{aligned} \right.
\end{equation*}

Case 3) $n\equiv6~(mod~8)$. The proof is similar to the proof of case 2).

Case 4) $n\equiv1~(mod~4)$. Now $e(M)=2-k$.~
From~(\ref{jf}),~Lemma~2.1,~Lemma~3.1~and Remark~2.2, we see that
$M$ admits an almost complex structure if and only if
\begin{equation*}
\left\{ \begin{aligned}
&\chi =0,\\
&\alpha =2a\xi,\\
&2a(n-1)! =2-k,
\end{aligned} \right.
\end{equation*}
where~$a\in\mathbb{Z}$. Hence by Lemma 3.1 and Lemma 2.1, $M$~admits
an almost complex structure if and only if~$M$~admits a stable
almost complex structure and
\begin{equation*}
2(n-1)! \mid (2-k).
\end{equation*}

Case 5) $n\equiv3~(mod~4)$. The proof is similar to the proof of case 4).
}
\end{proof}

\end{document}